\documentclass{amsart}
\usepackage{graphicx}
\usepackage{amsmath}
\usepackage{amsthm}
\usepackage{mathtext}
\usepackage{amssymb}
\usepackage{tikz}
\newtheorem{thm}{Theorem}
\newtheorem*{thm*}{Theorem}
\newtheorem{cor}{Corollary}
\newtheorem{lem}{Lemma}
\newtheorem{prop}{Proposition}
\newtheorem*{con*}{Conjecture}
\theoremstyle{definition}
\newtheorem{defn}{Definition}
\newtheorem*{ex*}{Example}
\newtheorem*{cons*}{Construction}
\newtheorem{rem}{Remark}

\DeclareMathOperator{\conv}{conv}
\DeclareMathOperator{\supp}{supp}

\begin{document}

\title{Cubical realizations of flag nestohedra and Gal's conjecture.}%
\author{Vadim Volodin}%
\address{Department of Mechanics and Mathematics, Moscow State University, Moscow, Russia}%
\email{volodinvadim@gmail.com}%

\thanks{Author is grateful to V.~M.~Buchstaber for statement of the problem and attention to the work and to N.~Erokhovets and A.~Gaifullin for useful discussions.}%

\thanks{The short version of the paper is accepted by the journal "Uspekhi Matematicheskikh Nauk"}%
\begin{abstract}
We study nestohedra $P_B$ corresponding to building sets $B$. It is shown that every flag nestohedron can be obtained from a cube by successive shavings faces of codimension 2. We receive new Delzant geometric realization of flag nestohedra. The main result of the paper is that Gal's conjecture holds for every flag nestohedron. Moreover, we get the exact estimation of $\gamma$-vectors of $n$-dimensional flag nestohedra: $0\leq\gamma_i(P_B)\leq\gamma_i(Pe^n)$.
\end{abstract}
\maketitle
\section{Introduction}
Simple polytopes are the polytopes in general position with respect to moving facets. The classical problem of describing $f$-vectors of simple polytopes was solved in works \cite{St}, \cite{BL}, \cite{Mc93} in terms of $h$-polynomials. A polytope is called flag if any collection of its pairwise intersecting faces has a nonempty intersection. The problem of describing $f$-vectors of flag polytopes is open. The Dehn-Sommeville equations give that $h_i=h_{n-i}$, so we can define the $\gamma$-vector and $\gamma$-polynomial from the expression $h(P)(t)=\sum_{i=0}^{[\frac{n}{2}]}\gamma_i t^i (1+t)^{n-2i}$.

The first conjecture about conditions on $\gamma$-vectors of simple flag polytopes is Charney-Davis conjecture (see \cite{CD}), which is equivalent to nonnegativity of $\gamma_{[\frac{n}{2}]}$. Later T.~Januszkiewicz predicted that the $h$-polynomial of any simple flag polytope has only real roots. That is strengthening of Charney-Davis conjecture. In \cite{Gal} it was shown that the real root conjecture fails in four and higher dimensions and was formulated (in dual form and more general view)
\begin{con*}[Gal, \cite{Gal}, 2005]
The $\gamma$-vector of any simple flag polytope has nonnegative entries.
\end{con*}
Nestohedra is a wide class of simple polytopes with well-described combinatorics. In \cite{PRW} Gal's conjecture was proved for nestohedra corresponding to chordal building sets. In works \cite{Er},\cite{Fenn} the conjecture was proved for nestohedra corresponding to complete biparitite graphs.

In \cite{FM} it was shown that if $B_1\subset B_2$, then $P_{B_2}$ can be obtained from $P_{B_1}$ by sequence of face shavings. Here we develop this idea and show that if $P_{B_1}$ and $P_{B_2}$ are flag, then we can change the order of shavings so that only faces of codimension 2 will be shaved off.

V.~M.~Buchstaber described realization of the associhedron $As^n\subset \mathbb{R}^n$ (see \cite{B1}, Theorem 5.1) as a polytope obtained from the standard cube by shavings faces of codimension 2. The main result of this paper is that every flag nestohedron has such realization. As a corollary we obtain
\begin{thm*}
The $\gamma$-vector of any flag nestohedron has nonnegative entries.
\end{thm*}
Particularly, Gal's conjecture holds for all graph-associhedra. Also we solve the problem stated in \cite{PRW}, which assumes that if $B_1\subset B_2$ and $P_{B_i}$ are flag, then $\gamma(P_{B_1})\leq\gamma(P_{B_2})$. That yields the higher bound for $\gamma$ vectors of flag nestohedra.

Applying the technique of shavings, we construct new geometric realization of a flag nestohedron as a Delzant polytope in $\mathbb{R}^n$.

It is well-known that every nestohedron is a Delzant polytope. By Delzant theorem for every Delzant polytope $P^n$ there exists a Hamiltonian toric manifold $M^{2n}$ such that $P^n$ is an image of moment map. Davis-Januszkiewicz theorem (see \cite{DJ}) states that odd Betti numbers of $M^{2n}$ are zero and even Betty numbers are equal to coordinates of the $h$-vector of $P^n$ ($b_{2i}(M)=h_i(P)$). So, Gal's conjecture is closely connected to differential geometry of Hamiltonian toric manifolds.

\section{Enumerative polynomials of polytopes}
Let $f_i$ be the number of $i$-dimensional faces of an $n$-dimensional polytope $P$. The vector $(f_0,f_1,\ldots ,f_{n-1},f_n)$ is called the $f$-vector of the polytope $P$, the polynomial $$f(P)(t)=f_0+f_1t+\ldots +f_{n-1}t^{n-1}+f_n t^n$$ is called the $f$-polynomial of the polytope $P$. The $h$-vector and $h$-polynomial are defined by:
$$h(t)=h_0+h_1t+\ldots +h_{n-1}t^{n-1}+h_n t^n=f_0+f_1(t-1)+\ldots +f_{n-1}(t-1)^{n-1}+f_n(t-1)^n=f(t-1).$$

It is also useful to define the $H$-polynomial (see \cite{B}) of two variables:
$$H(P)(\alpha,t)=h_0\alpha^n+h_1\alpha^{n-1}t+\dots+h_{n-1}\alpha t^{n-1}+h_n t^n.$$

The Dehn-Sommerville equations (see \cite{BP}) yield that $H(P)$ is symmetric for simple polytopes. Therefore, it can be represented as a polynomial of $a=\alpha+t$ and $b=\alpha t$:

\begin{align*}
H(P)=\sum\limits_{i=0}^{[\frac{n}{2}]}\gamma_i(\alpha t)^i(\alpha+t)^{n-2i}.\\
\intertext{Substitute $\alpha=1$ and obtain:}
h(P)=\sum\limits_{i=0}^{[\frac{n}{2}]}\gamma_i t^i(1+t)^{n-2i}.
\end{align*}

Both decompositions are unique. The vector $(\gamma_0,\gamma_1,\dots,\gamma_{[\frac{n}{2}]})$ is called the $\gamma$-vector of the polytope $P$, and the $\gamma$-polynomial is defined by $\gamma(P)(\tau)=\gamma_0+\gamma_1\tau+\dots+\gamma_{[\frac{n}{2}]}\tau^{[\frac{n}{2}]}$.

\section{Nestohedra}
In this section we state well-known facts about nestohedra.

\begin{defn}
A collection $B$ of nonempty subsets of $[n+1]=\{1,\ldots,n+1\}$ is called a \emph{building set} on $[n+1]$ if the following conditions hold:
\begin{enumerate}
\item[1)] $\{i\}\in B$ for all $i\in[n+1]$;
\item[2)] $S_1,S_2\in B\text{ and }S_1\cap S_2\neq\emptyset\Rightarrow S_1\cup S_2\in B$.
\end{enumerate}
The building set $B$ is \emph{connected} if $[n+1]\in B$.
\end{defn}

Two building sets $B_1$ and $B_2$ on $[n+1]$ are \emph{equivalent} if there exists a permutation $\sigma: [n+1]\to [n+1]$ that induces one to one correspondence $B_1\to B_2$

The \emph{restriction} of the building set $B$ to $S\subset [n+1]$ is the following building set on $[|S|]$: $$B|_S=\{S'\in B: S'\subset S\}$$


Define the \emph{product} of building sets $B_1$ and $B_2$ on $[n_1+1]$ and $[n_2+1]$ as the building set $B=B_1\cdot B_2=B_1\sqcup B_2$ on $[n_1+n_2+2]$ induced by connecting the interval $[n_1+1]$ to the interval $[n_2+1]$.

Restriction and product are defined up to equivalence between building sets.

The \emph{closure} of the set $B\subset [n+1]$ is the minimal by inclusion building set $\widehat{B}$ containing $B$.

\begin{defn}
Let $\Gamma$ be a graph with no loops or multiple edges on the node set $[n+1]$. The \emph{graphical building set} $B(\Gamma)$ is the collection of nonempty subsets $J\in[n+1]$ such that the induced subgraph $\Gamma|_J$ on the node set $J$ is connected.
\end{defn}

The convex $n$-dimensional polytope is called \emph{simple} if its every vertex is contained in exactly $n$ facets. The set of combinatorial simple polytopes has a structure of differential ring (see \cite{B}) introduced by V.~M.~Buchstaber.

The convex polytope is called \emph{flag} if any collection of its pairwise intersecting facets has a nonempty intersection.

The convex polytope $P\subset V, \dim V=n$ is called a \emph{Delzant polytope} if there exists a basis $A=\{a_i\}$ of $V$ such that for every vertex $v$ of $P$ there exist integer vectors parallel to outer normals to facets containing $v$ and forming a $\mathbb{Z}$ basis of $\mathbb{Z}^n\subset\mathbb{R}^n$, where $\mathbb{Z}^n$ is the set of vectors with integer coordinates. Here the scalar product and coordinates are defined by the basis $A$.

Let $M_1$ and $M_2$ be subsets of $\mathbb{R}^n$. The \emph{Minkowski sum} of $M_1$ and $M_2$ is the following subset of $\mathbb{R}^n$: $$M_1+M_2=\{x\in\mathbb{R}^n: x=x_1+x_2, x_1\in M_1, x_2\in M_2\}$$ If $M_1$ and $M_2$ are convex sets, then so is $M_1+M_2$. If $M_1$ and $M_2$ are convex polytopes, then so is $M_1+M_2$.

\begin{defn}
Let $e_i$ be the endpoints of the basis vectors of $\mathbb{R}^{n+1}$. Define the \emph{nestohedron} $P_B$ corresponding to the building set $B$ as following
$$P_B=\sum_{S\in B}\Delta^S\text{, where }\Delta^S=\conv\{e_i, i\in S\}.$$
If $B$ is a graphical building set, then $P_B$ is called a \emph{graph-associhedron}.
\end{defn}

\begin{ex*}
The building set $B=\{[n+1],\{i\}, i\in[n+1]\}$ corresponds to the simplex $\Delta^n$.

The building set $B=2^{[n+1]}=\{S: S\subset[n+1]\}$ corresponds to the permutohedron $Pe^n$.

The building set $B=\{[i,j], 1\leq i\leq j\leq n+1\}$, where $[i,j]=\{k:i\leq k\leq j\}$, corresponds to the associhedron $As^n$.
\end{ex*}

\begin{prop}\label{basic}
Let $B$ be a connected building set on $[n+1]$. Then the polytope $P_B$ has dimension $n$ and is the intersection of the hyperplane $H=\sum_{i=1}^{n+1}x_i=|B|$ with halfspaces $H_S=\sum_{i\in S}x_i\geq |B|_S|$, where $S\in B\setminus[n+1]$. Moreover, every inequality determines a facet of $P$.
\end{prop}

If $B$ is a connected building set on $[n+1]$, then we can identify facets of nestohedron $P_B$ with elements of $B\setminus[n+1]$. Denote by $F_S$ the facet of $P_B$ corresponding to the element $S\in B$.

\begin{prop}\label{criteria}
Let $B$ be a building set. Then facets $F_{S_1},\ldots ,F_{S_k}$ have a nonempty intersection if and only if the following conditions hold:
\begin{enumerate}
\item[1)] $\forall S_i, S_j: S_i\subset S_j\text{ or }S_i\supset S_j\text{ or }S_i\cap S_j=\emptyset$;
\item[2)] $\forall S_{i_1},\dots,S_{i_p}\text{ such that } S_{i_j}\cap S_{i_l}=\emptyset:  S_{i_1}\sqcup\ldots\sqcup S_{i_p}\notin B$.
\end{enumerate}
\end{prop}

\begin{prop}
Every nestohedron $P_B$ is a Delzant polytope in the basis $\{e_i-e_{n+1}, i=1,\ldots ,n\}$. Particularly, every nestohedron is simple.
\end{prop}

\begin{prop}
Every graph-associhedron is flag.
\end{prop}
So flag nestohedra is a wide class of flag simple polytopes that includes graph-associhedra.

\section{Cube shavings}

\textbf{Denotions} Elements $S\in B\setminus[n+1]$ are identified with facets of nestohedron. So, the expression "elements $S_i$ intersect" means that $S_i$ intersect as subsets of $[n+1]$. The expression "facets $S_i$ intersect" means that corresponding facets intersect. Also if some facets of one polytope $P_1$ are identified with some facets of the other polytope $P_2$ (for example, if $B_1\subset B_2$), we write "$F_i$ intersect in $P_1$ (or in $P_2$)". Facets and faces are usually denoted by $F$ and $G$.

\begin{cons*}(Face shaving)
Let $P\subset \mathbb{R}^n$ be a simple $n$-dimensional polytope such that $0$ is its internal point and $G$ be its face. Let $l_G\in{{\mathbb R}^n}^\ast$ be a linear function such that $l_G(P)\leq1$ and \\$\{x\in P: l_G x=1\}=G$. Let us call the simple polytope $Q=\{x\in P: l_G x\leq 1-\varepsilon\}$ be obtained from the polytope $P$ by shaving the face $G$. Here $\varepsilon$ is small enough such that all the vertexes of $P$ that don't belong to $G$ satisfy $l_G x\leq 1-\varepsilon$. Note that combinatorics of the polytope $Q$ depends only on combinatorics of the polytope $P$ and the shaved off face. The polytope $Q$ has one new facet determined by the section $l_G x=1-\varepsilon$.
\end{cons*}

\begin{prop}\label{shaveresult}
Let the polytope $Q$ be obtained from the simple $n$-dimensional polytope $P$ by shaving the face $G$ of dimension $k$. Then the new facet $F_0$ corresponding to the section of $P$ is combinatorially equivalent to $G\times \Delta^{n-k-1}$.
\end{prop}

\begin{proof}
Let $P'$ be the intersection of $P$ with the halfspace (determined by the section) that contains $G$. First, show that the polytope $P'$ is combinatorially equivallent to $G\times \Delta^{n-k}$. Let $G=F_1\cap\ldots\cap F_{n-k}$, where $F_i$ are facets of $P$ and $k=\dim G$. Facets of $G$ are in one to one correspondence with facets $F'_1,\ldots ,F'_l$ of the polytope $P'$, where $F'_1,\ldots ,F'_l$ intersect $G$ but don't contain it. Denote the section by $F_0$. All the facets of $P'$ are $F_0,F_1,\ldots ,F_{n-k}$ and $F'_1,\ldots ,F'_l$. Indeed, $P'$ has no other facets, since picking the section close enough to $G$ one can separate $G$ from the facets that don't intersect $G$. Every face of $P'$ not contained in $F_0$ intersects $G$ and every face of $P'$ not contained in $G$ intersects $F_0$. Therefore, we have
$$\bigcap_{i\in I}F_i\bigcap_{j\in J}F'_j\neq\emptyset\text{ in $P'$}\Leftrightarrow |I|\leq n-k\textrm{ and }\bigcap_{j\in J}(G\cap F'_j)\neq\emptyset\text{ in $P$}$$
Identifying facets $F_i$ with facets of $\Delta^{n-k}$, and facets $F'_i$ of $P'$ with facets $G\cap F'_i$ of $G$, we obtain that the face lattices of polytopes $P'$ and $G\times \Delta^{n-k}$ are equivalent.

The result follows, since $F_0$ is the facet of $G\times\Delta^{n-k}$ that doesn't intersect $G\times pt$.
\end{proof}

\begin{prop}
Let the polytope $Q$ be obtained from the simple $n$-dimensional polytope $P$ by shaving the face $G$ of dimension $k$, then $\gamma(Q)=\gamma(P)+\tau\gamma(G)\gamma(\Delta^{n-k-2})$.
\end{prop}
\begin{proof}
The shaving removes the face $G$ and adds the face $G\times\Delta^{d-k-1}$ instead, then $$f(Q)=f(P)+f(G)f(\Delta^{n-k-1})-f(G).$$ Whence:
\begin{multline*}
h(Q)=h(P)+h(G)h(\Delta^{n-k-1})-h(G)=h(P)+h(G)(\sum_{i=0}^{n-k-1}t^i-1)=h(P)+t h(G)h(\Delta^{n-k-2})=\\=\sum_{i=0}^{[\frac{n}{2}]}\gamma_i^Pt^i(t+1)^{n-2i}+
t\left(\sum_{i=0}^{[\frac{k}{2}]}\gamma_i^Gt^i(t+1)^{k-2i}\right)\left(\sum_{j=0}^{[\frac{n-k-2}{2}]}\gamma_j^{\Delta} t^j(t+1)^{n-k-2-2j}\right)=\\=
\sum_{i=0}^{[\frac{n}{2}]}\gamma_i^Pt^i(t+1)^{n-2i}+\sum_{i=0}^{[\frac{k}{2}]}\sum_{j=0}^{[\frac{n-k-2}{2}]}\gamma_i^G\gamma_j^{\Delta}t^{i+j+1}(t+1)^{d-2(i+j+1)}
\end{multline*}
The result follows.
\end{proof}

\begin{cor}
Let the polytope $Q$ be obtained from the simple polytope $P$ by shaving the face $G$ of codimension 2, then $\gamma(Q)=\gamma(P)+\tau\gamma(G)$.
\end{cor}

Consider the set $\mathcal{P}^{cube}$ of combinatorial polytopes that can be obtained from a cube by successive shavings faces of codimension 2. The dual operation, edge subdivision, used in \cite{Gal} with respect to flag simplitial polytopes.

\begin{lem}\label{flagshave}
After shaving a face of codimension 2 any simple flag polytope stays simple and flag.
\end{lem}
\begin{proof}
Let $Q$ be obtained from $P$ by shaving the face $G=F_1\cap F_2$. All facets of $Q$ are facets of $P$ and the new facet $F_0$. Suppose that facets $\mathcal{F}$ of $Q$ are pairwise intersecting. Note that $\mathcal{F}$ doesn't contain $\{F_1,F_2\}$. If $F_0\notin\mathcal{F}$, then facets $\mathcal{F}$ intersect in $P$, and their intersection is not contained in $G$. Then some part of the intersection stays in $Q$ after shaving. If $F_0\in\mathcal{F}$, then $G'=\bigcap_{F_i\in\mathcal{F}\setminus F_0}F_i$ is a nonempty face of $P$. Note that $G'$ intersects $G$ but is not contained in $G$, whence $G'$ intersects $F_0$.
\end{proof}

\begin{prop}\label{cubering}
If $P\in \mathcal{P}^{cube}$, then $\gamma_i(P)\geq 0$.
\end{prop}
\begin{proof}
Prove that if $P\in \mathcal{P}^{cube}$, then all its facets are in $\mathcal{P}^{cube}$ by induction on the number of shaved off faces. When nothing is shaved off, there is nothing to prove. Let $Q$ be obtained from $P\in \mathcal{P}^{cube}$ by shaving the face $G$ of codimension 2. Then the new facet has a form $F_0=G\times I\in \mathcal{P}^{cube}$ by the inductive assumption. Every other facet of $Q$ is obtained from some facet of $P$ by shaving some its face. Since every face of a flag polytope is flag, from proposition \ref{shaveresult} and lemma \ref{flagshave} it follows that codimension of the shaved off face can be 1 or 2.

Now, by induction on the dimension of $P$, using the formula $\gamma(Q)=\gamma(P)+\tau\gamma(G)$, we obtain the result.
\end{proof}

\section{Realizing nestohedra by shavings}

First show that every nestohedron corresponds to some connected building set.

\begin{cons*}[N.~Erokhovets]
Let $B,B_1,\ldots ,B_{n+1}$ be connected building sets on $[n+1],[k_1],\ldots ,[k_{n+1}]$. Define the connected building set $B(B_1,\ldots ,B_{n+1})$ on $[k_1]\sqcup\ldots\sqcup[k_{n+1}]=[k_1+\ldots +k_{n+1}]$ consisting of the elements $S^i\in B_i$ and $\sqcup_{i\in S}[k_i]$, where $S\in B$.
\end{cons*}

\begin{lem}[N.~Erokhovets]\label{eq1}
Let $B,B_1,\ldots ,B_{n+1}$ be connected building sets on $[n+1],[k_1],\ldots ,[k_{n+1}]$, and $B'=B(B_1,\ldots ,B_{n+1})$. Then $P_{B'}$ is combinatorially equivalent to $P_B\times P_{B_1}\times\dots\times P_{B_{n+1}}$.
\end{lem}
\begin{proof}
Consider $B''=B\sqcup B_1\sqcup\ldots\sqcup B_{n+1}$. Facets $S_1,\ldots ,S_l\in B\setminus[n+1]$ and $S_1^i,\ldots ,S_{l_i}^i\in B_i\setminus[k_i]$ intersect in $P_{B''}$ if and only if facets $S_1,\ldots ,S_l$ intersect in $P_B$, and for every $i$ facets $S_1^i,\ldots ,S_{l_i}^i$ intersect in $P_{B_i}$. The building set $B''$ is a product of the building sets, whence $P_{B''}\sim P_B\times P_{B_1}\times\dots\times P_{B_{n+1}}$.\\
Consider the mapping $\varphi: B''\to B'$, defined by
$$\varphi(S)=\begin{cases} S&\text{, if $S\in B_i$}\\\bigsqcup\limits_{i\in S}[k_i]&\text{, if $S\in B$}\end{cases}$$
As we can see, $\varphi$ determines a bijection between facets of $P_{B''}$ and facets of $P_{B'}$. By proposition \ref{criteria}, facets $\varphi(S_1),\ldots ,\varphi(S_k)$ intersect in $P_{B'}$ if and only if facets $S_1,\ldots ,S_k$ intersect in $P_{B''}$. Therefore, $P_{B'}\sim P_{B''}$.
\end{proof}

\begin{cor}
For every nestohedron $P$ there exists a connected building set $B$ such that $P_B$ is combinatorially equivalent to $P$.
\end{cor}

\begin{proof}
Indeed, an arbitrary building set $B'$ has a form $B_1\sqcup\ldots\sqcup B_k$, where $B_i$ are connected building sets. Let's set the building set $B''=B_1(B_2,\{1\},\ldots,\{1\})\sqcup B_3\sqcup\ldots\sqcup B_k$ that corresponds to the same nestohedron and has less maximal by inclusion elements than $B'$. Then apply the construction of substitution to $B''$ and so on. Finally we get a building set $B$ with the unique maximal by inclusion element. It follows that $B$ is connected.
\end{proof}

Without loss of generality, we suppose that every nestohedron corresponds to a connected building set.

In sequel we construct the sequence of building sets $B_0\subset\dots\subset B_N=B$, where $B_0$ corresponds to the cube. From \cite{FM} (Theorem 4.2) we can extract that for connected building sets $B'\subset B''$ the polytope $P_{B''}$ can be obtained from $P_{B'}$ by sequence of shavings. Our purpose is to show that if $P_B$ is flag, then we can choose $B_i$ such a way that every shaved off face has codimension 2. In this case, $P_B$ has nonnegative $\gamma$-polynomial, since $P_B\in \mathcal{P}^{cube}$.
First, we find the building subset $B_0\subset B$ such that $P_{B_0}$ is combinatorially equivalent to the cube $I^n$.

\begin{lem}[\cite{PRW}, Prop. 7.1]\label{Bcube}
If $B$ is a connected building set on $[n+1]$, and $P_B$ is flag, then there exists a connected building set $B_0\subset B$ such that $P_{B_0}$ is combinatorially equivalent to the cube $I^n$.
\end{lem}
\begin{proof}
Prove it by induction on $n$. If $n=1$, there is nothing to prove. Suppose that the lemma holds for all $m\leq n$ and prove it for $m=n+1$. Pick the maximal by inclusion collection $S_1,\ldots ,S_k\in B\setminus[n+1]$ such that $S_1\sqcup\ldots\sqcup S_k=[n+1]$. Note that $\forall J\subset[k], 1<|J|<k: \bigsqcup_{j\in J}S_j\notin B$. Therefore, $k=2$. Indeed, if $k>2$, then facets $S_1,\ldots ,S_k$ are pairwise intersecting, but their intersection is empty set. The building sets $B|_{S_1}$ and $B|_{S_2}$ also correspond to flag polytopes. By the inductive assumption, there exist $B_0^1\subset B|_{S_1}$ and $B_0^2\subset B|_{S_2}$ such that $P_{B_0^1}\sim I^{|S_1|-1}$ and $P_{B_0^2}\sim I^{|S_2|-1}$. Put $B_0=B_0^1\cup B_0^2\cup[n+1]$. By lemma \ref{eq1}, we have $P_{B_0}\sim I\times P_{B_0^1}\times P_{B_0^2}\sim I^{1+(|S_1|-1)+(|S_2|-1)}\sim I^n$.
\end{proof}

Now, let us show which faces will be shaved off in general construction.

\begin{cons*}[Decomposition of $S\in B_1$ by elements of $B_0$]
Let $B_0$ and $B_1$ be connected building sets on $[n+1]$, $B_0\subset B_1$, and $S\in B_1$. Let us call the decomposition of $S$ by elements of $B_0$ the representation $S=S_1\sqcup\ldots \sqcup S_k, S_j\in B_0$ such that $k$ is minimal among such disjoint representations. Denote the decomposition by $B_0(S)$.
\end{cons*}

The next proposition can be easily checked.

\begin{prop}\label{pr}
Decomposition has the following properties:
\begin{enumerate}
\item[1)] Decomposition exists and is unique.
\item[2)] If $B_0\subset B_1\subset B_2$, then for $S\in B_2$ we have $B_0(S)=B_0(B_1(S))$.
\item[3)] If $S'\subset S$, then $\exists S_j\in B_0(S): S'\subset S_j$.
\item[4)] Let $S=S_1\sqcup\ldots \sqcup S_k, S_j\in B_0$. The collection $\{S_j\}$ is the decomposition of $S$ if and only if $$\forall J\subset [k], 1<|J|<k:\bigsqcup_{j\in J}S_j\notin B_0$$
\end{enumerate}
\end{prop}

\begin{lem}\label{norm}
Let $P,Q\subset\mathbb{R}^n$ be simple polytopes with same outer normals $\mathcal{V}$ to facets, and for every $V\subset \mathcal{V}$ we have $\bigcap_{v\in V}P_v\neq\emptyset\Rightarrow\bigcap_{v\in V}Q_v\neq\emptyset$, where $P_v$ and $Q_v$ are the facets of $P$ and $Q$ corresponding to $v$. Then $P$ and $Q$ are combinatorially equivalent.
\end{lem}
\begin{proof}
Consider an arbitrary vertex $w$ of the polytope $P$. Let $P_{v_1},\ldots ,P_{v_n}$ be all the facets containing $w$. Corresponding facets $Q_{v_1},\ldots ,Q_{v_n}$ of the polytope $Q$ have a nonempty intersection, moreover, the face $Q_{v_1}\cap\ldots\cap Q_{v_n}$ doesn't intersect the other facets of $Q$, since $Q$ is simple. Therefore, every vertex of $P$ corresponds to some vertex of $Q$ with the same normal cone. Since the normal fan of $P$ is complete, $Q$ has no other vertexes, then normal fans of $P$ and $Q$ are the same. The result follows, since normal fan completely determines combinatorial type of a polytope.
\end{proof}

\begin{prop}\label{fac}
Let $Q$ be obtained from the simple polytope $P$ by shaving the face $G=F_1\cap\ldots\cap F_k$. Denote by $F_0$ the new facet corresponding to the section. Then a collection $\mathcal{F}$ of facets of $Q$ intersects, if and only if one of the following conditions holds:
\begin{enumerate}
\item[a)] $F_0\notin\mathcal{F}, \{F_1,\ldots ,F_k\}\nsubseteq\mathcal{F}$ and $\bigcap_{F_i\in\mathcal{F}}F_i\neq\emptyset$ in $P$;
\item[b)] $F_0\in \mathcal{F}, \{F_1,\ldots ,F_k\}\nsubseteq\mathcal{F}$ and $\bigcap_{F_i\in\mathcal{F}\setminus F_0}F_i\cap G\neq\emptyset$ in $P$.
\end{enumerate}
\end{prop}
\begin{proof}
Indeed, condition a) means that facets $\mathcal{F}$ intersect in $P$, and their intersection is not contained in $G$. Since the section is in a small neighborhood of $G$, a part of the intersection stays in $Q$ after shaving. Condition b) means that facets $\mathcal{F}\setminus F_0$ intersect in $P$, and their intersection $\bigcap_{F\in \mathcal{F}\setminus F_0}F_i$ intersects $G$, but is not contained in $G$. Therefore, $\bigcap_{F\in \mathcal{F}\setminus F_0}F_i$ intersects $F_0$.
\end{proof}

\begin{cons*}[Polytope $P_{cut}$]
Let $B_0$ and $B_1$ be connected building sets on $[n+1]$, and $B_0\subset B_1$. The set $B_1$ is partially ordered by inclusion. Let us number all the elements of $B_1\setminus B_0$ by indexes $i$ such a way that $i\leq i'$ provided $S^i\supseteq S^{i'}$. By definition, let's set $P_{cut}$ be the polytope obtained from $P_{B_0}$ by successive shavings faces $G^i=\bigcap_{j=1}^{k_i} F_{S_j^i}$ that correspond to $S^i=S_1^i\sqcup\ldots\sqcup S_{k_i}^i\in B_1\setminus B_0$, starting from $i=1$ (i.e., maximal by inclusion element). It is well defined by proposition \ref{fac}, i.e., facets $S_1^i,\ldots ,S_{k_i}^i$ corresponding to the decomposition of some element of $B_1\setminus B_0$ intersect until their intersection will be shaved off.
\end{cons*}

\begin{lem}\label{main}
Let $B_0$ and $B_1$ be connected building sets on $[n+1]$, and $B_0\subset B_1$. Then $P_{cut}\sim P_{B_1}$.
\end{lem}
\begin{proof}
Prove the lemma by induction on $N=|B_1|-|B_0|$.

Let $N=1$, then $B_1=B_0\cup S^1$. Define the facet correspondence between $P_{cut}$ and $P_{B_1}$: $S\in B_0$ corresponds to $S\in B_1$, the facet obtained by shaving $G^1=\bigcap_{j=1}^{k_1}F_{S_j^1}$ corresponds to $S^1=\bigsqcup_{j=1}^{k_1} S_j^1\in B_1$. Consider the standard geometric realization of $P_{B_0}$ and $P_{B_1}$, described in proposition \ref{basic}. Shaving the face corresponding to $S^1=S_1^1\sqcup\ldots\sqcup S_{k_1}^1$ is equivalent to adding new inequality $\sum_{i\in S^1}x_i\geq\sum_{j=1}^{k_1}|B_0|_{S_j^1}|+\varepsilon$, where $\varepsilon$ is a small positive number. Then polytopes $P_{B_1}$ and $P_{cut}$ have same outer normals to their facets. By lemma \ref{norm}, it is enough to prove that if some collection of facets intersects in $P_{B_1}$, then corresponding facets intersect in $P_{cut}$. Assume that facets $\mathcal{S}$ intersect in $P_{B_1}$ and show that they intersect in $P_{cut}$. First, note that $\{S_1^1,\ldots ,S_{k_1}^1\}\nsubseteq\mathcal{S}$. If $S^1\notin\mathcal{S}$, then facets $\mathcal{S}$ intersect in $P_{B_0}$ and, by a) proposition \ref{fac}, they intersect in $P_{cut}$. If $S^1\in\mathcal{S}$, then for every element $S_i\in\mathcal{S}$ intersecting $S^1$ either $S^1\subset S_i$ or $\exists j: S_i\subset S_j^1$. Therefore, $\bigcap_{S_i\in\mathcal{S}\setminus S^1}F_{S_i}\cap G^1=\bigcap_{S_i\in\mathcal{S}\setminus S^1}F_{S_i}\bigcap_{j=1}^{k_1}F_{S_j^1}\neq\emptyset$ in $P_{B_0}$, and, by b) proposition \ref{fac}, facets $\mathcal{S}$ intersect in $P_{cut}$.

Assume that the result holds for $M<N$ and prove it for $M=N$. Let $|B_1|-|B_0|=N$, then $P_{cut}$ is obtained from $P_{B_0}$ by successive shavings faces corresponding to $S^i=S_1^i\sqcup\ldots\sqcup S_{k_i}^i, i=1,\ldots,N$ in reverse inclusion order. Therefore, $B'_0=B_0\cup S^1$ is a building set. By the inductive assumption, $P_{B'_0}$ is obtained from $P_{B_0}$ by shaving the face corresponding to $S^1$, and $P_{B_1}$ is obtained from $P_{B'_0}$ by successive shavings the faces corresponding to $S^i, i=2,\ldots ,N$. Whence, $P_{cut}\sim P_{B_1}$.
\end{proof}

\begin{lem}\label{step}
Let $B_1$ and $B_3$ be connected building sets on $[n+1]$, $B_1\subsetneq B_3$, and polytopes $P_{B_1}$ and $P_{B_3}$ be flag. Then $\exists B_2: B_1\subsetneq B_2\subseteq B_3$ such that $P_{B_2}$ is obtained from $P_{B_1}$ by successive shavings faces of codimension 2.
\end{lem}
\begin{proof}
Pick $B_2=\widehat{B_1\cup S}$, where $S$ is the minimal by inclusion element of $B_3\setminus B_1$. Since $S\in B_3$ and $P_{B_3}$ is flag, there exist $I,J\in B_3: I\sqcup J=S$. From the chose of $S$ we have $I,J\in B_1$. It is easy to check that $B_1\cup\{S'=S_1\sqcup S_2, S_i\in B_1, I\subset S_1, J\subset S_2\}$ is the minimal building set containing $B_1\cup S$. Then, the decomposition of any element of $B_2\setminus B_1$ consists of exactly two elements. Therefore, obtaining $P_{B_2}$ from $P_{B_1}$ only faces of codimension 2 will be shaved off.
\end{proof}
\begin{rem}
By lemma \ref{flagshave}, the polytope $P_{B_2}$ is flag.
\end{rem}

\begin{thm}
If $B$ is a connected building set, and $P_B$ is flag, then
\begin{enumerate}
\item[1)] $P_B$ can be obtained from $I^n$ by successive shavings faces of codimension 2.
\item[2)] $Pe^n$  can be obtained from $P_B$ by successive shavings faces of codimension 2.
\end{enumerate}
\end{thm}
\begin{proof}
Pick $B_0\subset B$ such that $P_{B_0}$ is equivalent to $I^n$, then $B_0\subset B\subset 2^{[n+1]}$. Iterating lemma \ref{step}, we get the sequence of building sets $B_0\subset\dots\subset B_N=B\subset\dots\subset 2^{[n+1]}$ and finish the proof.
\end{proof}

Polynomials $\gamma(Pe^n)$ satisfy the differential equation (see \cite{B}), whence follows a simple recursion on $\gamma_i(Pe^n)$ that particularly gives their nonnegativity. The next theorems follow from obtained results.
\begin{thm}
For any flag $n$-dimmensional nestohedron $P_B$ we have $0\leq\gamma_i(P_B)\leq\gamma_i(Pe^n)$.
\end{thm}

\begin{thm}
If $B_1$ and $B_2$ are connected building sets on $[n+1]$, $B_1\subset B_2$, and $P_{B_i}$ are flag, then $\gamma_i(P_{B_1})\leq\gamma_i(P_{B_2})$.
\end{thm}

We conclude that Gal's conjecture holds for all nestohedra.

\section{Geometric realization of flag nestohedra in $\mathbb{R}^n$}
Now, let's realize an arbitrary flag nestohedron $P_B$ from the standard cube in $\mathbb{R}^n$ by sequence of shavings.

Let $B_0\subset B$ be the building set corresponding to the cube $I^n$. Identify elements $S\in B_0\setminus[n+1]$ with inequalities $l_S x\leq b_S$ such that $l_S=\pm e_i, b_S=1$. Here $\mathbb{R}^n$ and ${\mathbb{R}^n}^\ast$ are identified with respect to the scalar product.

Now construct a realization of $P_B$ from this realization of the cube $P_{B_0}$. Let $B_0\subset B_1\subset\dots\subset B_N=B$ such that $B_i=\widehat{B_{i-1}\cup S^i}$, where $S^i$ are chosen as in lemma \ref{Bcube}, minimal by inclusion in $B\setminus B_{i-1}$. Order the elements of each $B_i\setminus B_{i-1}$ reversing inclusion, and set each element of $B_i\setminus B_{i-1}$ higher than all the elements of $B_{i-1}$. So we number all $S_j\in B\setminus B_0$ starting from $j=1$. Define inequalities for $P_B$ recursively. Every $S_j\in B_i\setminus B_{i-1}$ has a unique decomposition by elements of $B_{i-1}$: $S_j=S_{j_1}\sqcup S_{j_2}$.
\begin{align}
\intertext{By definition set:}
l_{S_j}&=l_{S_{j_1}} + l_{S_{j_2}}\\
b_{S_j}&=b_{S_{j_1}} + b_{S_{j_2}} - \varepsilon_j
\end{align}
Here $\varepsilon_j>0$ is picked small enough so that adding the inequality $l_{S_j}x\leq b_{S_j}$ determines shaving the face $G^j=F_{S_{j_1}}\cap F_{S_{j_2}}$.

By the decomposition property $B_0(S)=B_0(B_1(S)), S\in B_1$, we get the explicit formula for $l_S$:
$$l_S=\sum_{j=1}^k l_{S_j},\text{ where $S=S_1\sqcup\ldots\sqcup S_k$ is the decomposition of $S$ by elements of $B_0$}$$
So, we can calculate important for toric topology matrix of outer normals.

\begin{prop}
Coordinates of the vectors $l_S$ are $0,\pm1$.
\end{prop}
\begin{proof}
Show that for $S',S''$: $S'\sqcup S''\notin B$ is fulfilled $\supp l_{S'}\cap\supp l_{S''}=\emptyset$, and the result will follow from (1). From the given construction we have the sequence of building sets $B_0\subset\ldots\subset B_M=B$, where $B_j=B_{j-1}\cup S_j$ (not $\widehat{B_{i-1}\cup S^i}$, as above!). The proof is by induction on the index $j$ of $B_j$. For the standard cube $P_{B_0}$ it is true, since facets $S'$ and $S''$ intersect. Assume that it is true for $B_{j-1}$ and prove it for $B_j$. The property holds for all the elements of $B_{j-1}\subset B_j$. Check the property for the new element $S_j$. Let $S_j=S_{j_1}\sqcup S_{j_2}$, where $S_{j_1},S_{j_2}\in B_{j-1}$. Then, if $S_j\sqcup S\notin B_j, S\in B_{j-1}$, then $S_{j_1}\sqcup S\notin B_{j-1}$ and $S_{j_2}\sqcup S\notin B_{j-1}$, then, by the inductive assumption, $\supp l_{S_j}\cap\supp l_{S}=(\supp l_{S_{j_1}}\cup\supp l_{S_{j_2}})\cap\supp l_{S}=\\=(\supp l_{S_{j_1}}\cap\supp l_S)\cup(\supp l_{S_{j_2}}\cap\supp l_S)=\emptyset$.
\end{proof}

\begin{prop}
Described realization of $P_B$ is Delzant in the standard basis of $\mathbb{R}^n$.
\end{prop}
\begin{proof}
Show that vectors $l_{S_{j_1}},\ldots ,l_{S_{j_n}}$ form a $\mathbb{Z}$ basis of $\mathbb{Z}^n$ provided $S_{j_1},\ldots ,S_{j_n}$ intersect in $P_B$. Prove it by induction on the number of shaved off faces or equivalent added inequalities. For the standard cube it is true. Let on the step $j$ the polytope $P_j$ be obtained from $P_{j-1}$ by adding inequality $l_{S_j}x\leq b_j$, where $S_j=S_{j_1}\sqcup S_{j_2}, l_{S_j}=l_{S_{j_1}} + l_{S_{j_2}}, b_{S_j}=b_{S_{j_1}} + b_{S_{j_2}} - \varepsilon_j$. By the inductive assumption, vectors $l_{S_{j_1}},\ldots ,l_{S_{j_n}}$ form a $\mathbb{Z}$ basis of $\mathbb{Z}^n$ provided facets $F_{S_{j_1}},\ldots ,F_{S_{j_n}}$ intersect in $P_{j-1}$. The new vertexes of $P_j$ are intersections of facets $F_{S_j},F_{S_{j_1}},F_{S_{j_3}},\ldots ,F_{S_{j_n}}$ and $F_{S_j},F_{S_{j_2}},\ldots ,F_{S_{j_n}}$ such that facets $F_{S_{j_1}},F_{S_{j_2}},\ldots ,F_{S_{j_n}}$ intersect in $P_{j-1}$. Therefore,
\begin{multline*}
\det(l_{S_j},l_{S_{j_1}},l_{S_{j_3}},\ldots ,l_{S_{j_n}})=\det(l_{S_{j_1}}+l_{S_{j_2}},l_{S_{j_1}},l_{S_{j_3}},\ldots ,l_{S_{j_n}})=\\
=\det(l_{S_{j_2}},l_{S_{j_1}},l_{S_{j_3}},\ldots ,l_{S_{j_n}})=-\det(l_{S_{j_1}},l_{S_{j_2}},l_{S_{j_3}},\ldots ,l_{S_{j_n}})=\pm1.
\end{multline*}
The second case is similar.
\end{proof}

\begin{ex*}
Let us realize the regular 3-dimensional associhedron. Its building set is \\ $B=\{\{1\},\{2\},\{3\},\{4\},\{1,2\},\{2,3\},\{3,4\},\{1,2,3\},\{2,3,4\},\{1,2,3,4\}\}$. The building set $B_0\subset B$ that gives a cube consists of $\{i\},\{1,2\},\{3,4\},[4]$. On the first step, with respect to lemma \ref{step}, we pick $S=\{2,3\}\in B\setminus B_0$ and obtain the building set $B=B_1=\widehat{B_0\cup S}$. By lemma \ref{main}, associhedron $P_B$ is obtained from the cube $P_{B_0}=I^3$ by shavings faces $F_{\{1,2\}}\cap F_{\{3\}}, F_{\{2\}}\cap F_{\{3,4\}}, F_{\{2\}}\cap F_{\{3\}}$ in the written order.

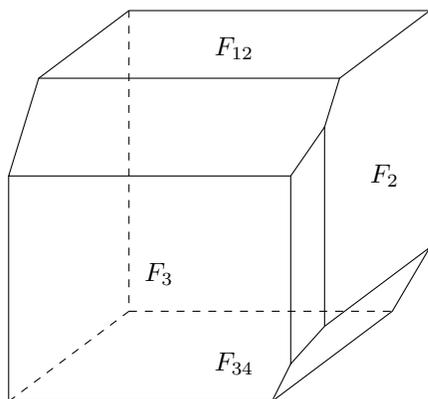
\begin{tikzpicture}[scale=0.1]
    \tikzstyle{ann} = [inner sep=2pt]
   \draw(5, 35)--(5, 5)--(40, 5);
   \draw(61,25.5)--(61,57)--(49,48);
   \draw(9, 48)--(21, 57)--(61,57);
   \draw[dashed](5, 5)--(21, 17)--(21,57);
   \draw[dashed](21, 17)--(56,17);
   \draw(47,41.5)--(49,48)--(9,48)--(5,35)--(42.5,35);
   \draw(42.5,10)--(40, 5)--(56,17)--(61,25.5)--(47, 15);
   \draw(47,41.5)--(47, 15)--(42.5,10)--(42.5,35)--(47,41.5);
    \node[ann] at (35, 52) {$F_{12}$};
    \node[ann] at (35, 10){$F_{34}$};
    \node[ann] at (25, 22){$F_3$};
    \node[ann] at (55, 35){$F_2$};
\end{tikzpicture}
Here $P_{B_0}$ is the standard cube $I^3$. Its left and right facets are $F_{\{1\}}$ and $F_{\{2\}}$, its front and back facets are $F_{\{3\}}$ and $F_{\{4\}}$, its top and bottom facets are $F_{\{1,2\}}$ and $F_{\{3,4\}}$. The top and bottom sections are $F_{\{1,2,3\}}$ and $F_{\{2,3,4\}}$. The vertical section is $F_{\{2,3\}}$.
\end{ex*}

\bibliographystyle{amsplain}

\end{document}